\documentclass[12pt]{amsart}
\usepackage{amsmath}
\usepackage{amssymb}
\usepackage[all,cmtip]{xy}

\DeclareMathAlphabet{\mathpzc}{OT1}{pzc}{m}{it}
\newcommand{\ncom}{\newcommand}

\ncom{\rar}{\rightarrow}
\ncom{\imply}{\Rightarrow}
\ncom{\lrar}{\longrightarrow}
\ncom{\into}{\hookrightarrow}
\ncom{\onto}{\twoheadrightarrow}
\ncom{\ov}{\overline}
\ncom{\m}{\mbox}
\ncom{\sta}{\stackrel}
\ncom{\invlim}{\varprojlim}
\ncom{\xhat}{\widehat}

\ncom{\vspc}{\vspace{3mm}}
\ncom{\End}{{\cE}nd}
\ncom{\tensor}{\otimes}

\ncom{\al}{\alpha}
\ncom{\cHom}{{\mathcal Hom}}

\ncom{\A}{{\mathbb A}}
\ncom{\comx}{{\mathbb C}}
\ncom{\E}{{\mathbb E}}
\ncom{\F}{{\mathbb F}}
\ncom{\G}{{\mathbb G}}
\ncom{\K}{{\mathbb K}}
\ncom{\Le}{{\mathbb L}}
\ncom{\N}{{\mathbb N}}

\ncom{\p}{{\mathbb P}}
\ncom{\Q}{{\mathbb Q}}
\ncom{\R}{{\mathbb R}}
\ncom{\Z}{{\mathbb Z}}

\ncom{\f}{\dfrac}

\ncom{\wtil}{\widetilde}

\ncom{\ci}{{\mathpzc i}}

\ncom{\cA}{{\mathcal A}}
\ncom{\cC}{{\mathcal C}}
\ncom{\cE}{{\mathcal E}}
\ncom{\cF}{{\mathcal F}}
\ncom{\cG}{{\mathcal G}}
\ncom{\cH}{{\mathcal H}}
\ncom{\cI}{{\mathcal I}}
\ncom{\cJ}{{\mathcal J}}
\ncom{\cL}{{\mathcal L}}
\ncom{\cM}{{\mathcal M}}
\ncom{\cN}{{\mathcal N}}
\ncom{\cO}{{\mathcal O}}
\ncom{\cP}{{\mathcal P}}
\ncom{\cQ}{{\mathcal Q}}
\ncom{\cR}{{\mathcal R}}
\ncom{\cS}{{\mathcal S}}
\ncom{\cT}{{\mathcal T}}
\ncom{\cU}{{\mathcal U}}
\ncom{\cV}{{\mathcal V}}
\ncom{\cW}{{\mathcal W}}
\ncom{\cX}{{\mathcal X}}
\ncom{\cY}{{\mathcal Y}}
\ncom{\cZ}{{\mathcal Z}}

\ncom{\cSU}{{\mathcal S \mathcal U}}
\ncom{\eop}{{\hfill $\Box$}}
\ncom{\isom}{\cong}

\ncom{\todo}{{\textbf{TODO}}}

\newtheorem{theorem}{Theorem}[section]
\newtheorem{lemma}[theorem]{Lemma}
\newtheorem{corollary}[theorem]{Corollary}

\newtheorem{definition}[theorem]{Definition}

\newtheorem{question}[theorem]{Question}

\newtheorem{answer}[theorem]{Answer}

\begin{document}
\baselineskip=16pt

\title{Splitting of low rank ACM bundles on hypersurfaces of high dimension}
\author{Amit Tripathi}
\address{Department of Mathematics, 
Indian Statistical Institute, Bangalore - 560059, India}
\email{amittr@gmail.com}
\thanks{This work was supported by a postdoctoral fellowship from NBHM, Department of Atomic Energy.}


\begin{abstract} Let $X$ be a smooth projective hypersurface. We derive a splitting criterion for arithmetically Cohen-Macaulay bundles over $X$. As an application we show that any rank 3 ACM vector bundle over $X$ splits when $\text{dim}\,X \geq 7$. We also derive a splitting result for rank 4 arithmetically Cohen-Macaulay bundles.
\end{abstract}
\maketitle
%%%%%%%%%%%%%%%%%%%%%%%%%%%%%%%%%%%%%%%%%%%%%%%%%%%%%%%%%%%%%%%%%%%%%%%
\section{Introduction}
%%%%%%%%%%%%%%%%%%%%%%%%%%%%%%%%%%%%%%%%%%%%%%%%%%%%%%%%%%%%%%%%%%%%%%%

Let $X \subset \p^{n+1}$ be a smooth hypersurface where $n \geq 3$. We set a conventional notation $$H^i_*(X, \cF) := \bigoplus_{m \in \Z} H^i(X, \cF(m))$$ where $\cF$ denotes a coherent sheaf on $X$. 

By Grothendieck-Lefschetz theorem \cite{SGA7II}, we know the structure of the set of all line bundles on $X$. Vector bundles over $X$ are not so well understood. An obvious question about vector bundles on any projective variety is the splitting problem - When can we say that a given vector bundle is a direct sum of line bundles? The proper objects in the category of vector bundles over $X$ to look for the splitting behaviour are arithmetically Cohen-Macaulay bundles. We recall the definition,
\begin{definition}
An \textit{arithmetically Cohen-Macaulay} (ACM) bundle on $X$ is a vector bundle $E$ satisfying $$H^i(X, E(m)) = 0, \,\, \forall \, m \in \Z \,\,\text{and} \,\,\, 0 < i < \text{dim}\, X$$
\end{definition} 

One can easily check that split bundles on a hypersurface are ACM bundles. The importance of this definition lies in a well known criterion of Horrocks \cite{Hor} - ACM bundles are precisely the bundles on $\p^n$ that are split. Viewing $\p^n$ as a hypersurface of degree 1 in $\p^{n+1}$, one may ask if for hypersurfaces with degree $d > 1$, such a splitting holds. 

When $d > 1$, there exists indecomposable ACM bundles on hypersurfaces (see \cite{M-R-R} for a specific example or \cite{M-R-R3} for a class of examples), though several splitting results are available for various degrees and ranks. In particular, fixing $d=2$, the ACM bundles on quadrics have been completely classified, see \cite{Kno}. The case of cubic surfaces in $\p^3$ has been investigated in \cite{C-H}.

In a different direction, we can fix the rank of the bundle and let degree vary.  Here the conjectural picture is that any ACM bundle of a fixed rank, over a sufficiently high dimensional hypersurface (irrespective of its degree) is split. The precise conjecture is,

\begin{proof}[Conjecture (Buchweitz, Greuel and Schreyer \cite{BGS}):] Let $X \subset \p^n$ be a hypersurface. Let $E$ be an ACM bundle on $X$. If rank $E < 2^e$, where $e = \left [ \displaystyle{\frac{n-2}{2}} \right ]$, then $E$ splits. (Here $[q]$ denotes the largest integer $\leq q$.) 
\end{proof} 

Splitting of ACM bundles of rank 2 on hypersurfaces have been understood fairly well. We summarize the results known. When $d = 1$, splitting follows by the  Horrock's criterion, so we assume $d \geq 2$. Let $E$ be a rank 2 ACM bundle on $X$, then $E$ splits if,

\begin{enumerate}
\item If dim$(X) \geq 5$ (see \cite{Kle} and \cite{M-R-R}).
\item If dim$(X) = 4$ and $X$ is general hypersurface and $d \geq 3$ (see \cite{M-R-R} and \cite{R}). 
\item If dim$(X) = 3$ and $X$ is general hypersurface and $d \geq 6$ (see \cite{M-R-R2} and \cite{R}).
\end{enumerate}

The case of a general hypersurface of low degree in $\p^4$ and $\p^5$ have also been studied by Chiantini and Madonna in \cite{C-M1}, \cite{C-M2}, \cite{C-M3}. 

For rank $ \geq 3$, very few results are known. A result in this direction is by Tadakazu \cite{S} who found a splitting criterion for ACM bundles on a general hypersurface depending on the degree of the hypersurface along with rank and dimension. Fujita \cite{Fuj} proved that any ACM vector bundle satisfying $H^2_*(X, \cE nd(E)) = 0$ splits. 

Here we prove the following splitting criterion for any hypersurface (irrespective of its degree) and a rank $k$ ACM bundle,
\begin{theorem}
Let $E$ be any rank $k$ bundle on a smooth hypersurface $X \subset \p^{n+1}$ with $n \geq 2k + 1$. Assume further that $E$ satisfies the following two conditions,
\begin{enumerate}
\item $H^i_*(X, E) = 0, \,\, i \in \{1, 2, \ldots ,n-1\}$
\item $H^i_*(X, \wedge^m E) = 0, \,\, i = 2m-1,2m,\ldots,k+m \,\, \text{for each} \,\,m \in \{2,\ldots,k-1\}  $
\end{enumerate}

Then $E$ splits.
\end{theorem}

The conjecture mentioned above predicts that any ACM bundle of rank 3 (resp. rank 4) over a hypersurface in $\p^6$ (resp. $\p^8$) splits. As a corollary to the above mentioned theorem, we prove:
\begin{theorem}[Corollary \ref{cor_1} + Corollary \ref{cor_2}] \label{result_main} Let $E$ be an ACM bundle on a smooth hypersurface $X \subset \p^{n+1}$. Then $E$ splits if,
\begin{enumerate}
\item rank $E = 3$ and $\text{dim}(X) \geq 7$.
\item rank $E = 4$, $\text{dim}(X) \geq 9$ and $E$ (its dual or any of its twists) admits a section with zero locus a complete intersection on $X$ of codimension 4.
\end{enumerate}
\end{theorem}

By a result of Kleiman \cite{Kleim}, one knows that the zero locus of any generic section of a rank $k$ vector bundle is a locally complete intersection (infact nonsingular) of codimension $k$ . For the part (2) of the Theorem, we want any section which corresponds to (global) complete intersection. 

For rank 2 ACM bundles, our method gives another proof for splitting when $n \geq 5$. 

%%%%%%%%%%%%%%%%%%%%%%%%%%%%%%%%%%%%%%%%%%%%%%%%%%%%%%%%%%%%%%%%%%%%%%%
\section{Preliminaries}
%%%%%%%%%%%%%%%%%%%%%%%%%%%%%%%%%%%%%%%%%%%%%%%%%%%%%%%%%%%%%%%%%%%%%%%

We will work over an algebraically closed field of characteristic zero. Let $X \subset \p^{n+1}$ be a hypersurface of degree $d \geq 2$. Let $E$ be a rank $k$ ACM bundle on $X$. We take a minimal (1-step) resolution of $E$ on $\p^{n+1}$,
\begin{eqnarray} \label{E_1_step}
0 \rar \wtil{F_1} {\rar} \wtil{F_0} \rar E \rar 0
\end{eqnarray}

where $\wtil{F_0}$ is direct sum of line bundles on $\p^{n+1}$. By Auslander-Buchsbaum formula, $\wtil{F_1}$ is a bundle and by Horrock's criterion it is also a split bundle on $\p^{n+1}$. 

Restricting \eqref{E_1_step} to $X$, we get,
\begin{eqnarray} \label{E_X}
0 \rar Tor^1_{\p^{n+1}}(E, \cO_X) \rar F_1 \rar F_0 \rar E \rar 0
\end{eqnarray}

where $F_i = \wtil{F_i} \otimes \cO_X$ for $i = 0,1$. To compute the \textit{Tor} term, we tensor the short exact sequence $0 \rar \cO_{\p^{n+1}}(-d) \rar \cO_{\p^{n+1}} \rar \cO_X \rar 0$ with $E$,
$$
0 \rar Tor^1_{\p^{n+1}}(E, \cO_X) \rar E(-d) \rar E \rar E \otimes \cO_{X} \rar 0
$$
The map $E \rar E \otimes \cO_{X}$ is an isomorphism, thus we get $Tor^1_{\p^{n+1}}(E, \cO_X) \cong E(-d)$. Exact sequence \eqref{E_X} breaks up into 2 short exact sequences,
\begin{eqnarray} \label{E_X_1}
0 \rar G \rar F_0 \rar E \rar 0
\end{eqnarray}
\begin{eqnarray} \label{E_X_2}
0 \rar E(-d) \rar F_1 \rar G \rar 0
\end{eqnarray}
Since $H^0_*(X, F_0) \onto H^0_*(X, E)$ is a surjection of graded rings, $H^1_*(X, G) = 0$. It follows that $G$ is also ACM. 

%%%%%%%%%%%%%%%%%%%%%%%%%%%%%%%%%%%%%%%%%%%%%%%%%%%%%%%%%%%%%%%%%%%%%%%
\section{Proof of the main results}
%%%%%%%%%%%%%%%%%%%%%%%%%%%%%%%%%%%%%%%%%%%%%%%%%%%%%%%%%%%%%%%%%%%%%%%

\begin{lemma} \label{G_direct_summand} Let $E$ be any non-split bundle (not necessarily ACM) on a hypersurface $X \subset \p^{n+1}, \, n \geq 3$. Assume further that $H^1_*(X, E^{\vee}) = 0$. 
Let the exact sequence \eqref{E_X_1} be a minimal (1-step) resolution of $E$ on $X$, then $G$ does not admit a line bundle as a direct summand.
\end{lemma} 
\begin{proof} We will assume the contrary. Let $G = G' \oplus L$ where $L$ is a line bundle. By Grothendieck-Lefschetz theorem, $L$ is of the form $\cO_X(a)$. There exists following pushout diagram,
\begin{equation} \label{dgm_G_direct_summand}
\begin{gathered} \xymatrix{0 \ar[r]& G' \ar[r] & F'_0 \ar[r] & E \ar[r]& 0 \\ 0 \ar[r] & G \ar[r] \ar[u] & F_0 \ar[r] \ar[u]& E \ar[r] \ar@{=}[u]& 0&
}\end{gathered}
\end{equation}
where $G \rar G'$ is the natural projection and $F'_0$ is the pushout. Completion of the diagram \eqref{dgm_G_direct_summand} gives $\eta: 0 \rar L \rar F_0 \rar F'_0 \rar 0$. Applying $Hom_X(-, L)$ to the top horizontal sequence gives, $$
\cdots \rar Ext^1(E, L) \rar Ext^1(F'_0, L) \rar Ext^1(G', L) \rar \cdots
$$
In the above sequence $\eta \mapsto \eta'$ where $\eta': 0 \rar L \rar G \rar G' \rar 0$ is split. By assumption $Ext^1(E, L) \cong H^1(X, E^{\vee} \otimes L) = 0$, thus $\eta$ splits. Therefore $F'_0$ is a direct sum of line bundles of rank $r-1$, by Krull-Schmidt theorem \cite{At}.  This implies that $0 \rar G' \rar F'_0 \rar E \rar 0$ is a 1-step resolution of $E$ which contradicts the minimality of the resolution \eqref{E_X_1}. 
\end{proof}

On any projective variety $Z$, for a short exact sequence of vector bundles $0 \rar \cF_0 \rar \cF_1 \rar \cF_2 \rar 0$ and any positive integer $k$, there exists a resolution of $k$-th symmetric power  of $\cF_2$, 
$$0 \rar \wedge^k(\cF_0) \rar \cdots \rar \wedge^{k-i}(\cF_0) \otimes Sym^i \cF_1 \rar \cdots Sym^k \cF_1 \rar Sym^k \cF_2 \rar 0$$

We will call this resolution the \textit{$\wedge-Sym$ sequence of index $k$} \footnote[1]{We were unable to find any standard terminology in the literature for the given resolution.} associated to the given short exact sequence. In fact, for any map $\phi: F_0 \rar F_1$ of free $R$-modules (where $R$ is a commutative ring), one considers $\cS$ the symmetric algebra on $F_1$. Fix a free basis of $F_1$ and assign degree 1 to the elements of the basis. Let $F' = \cS \otimes F_0(-1)$ (as $\cS$-modules) then there exists a natural $\cS$-module morphism of degree 0, $\phi': F' \rar \cS$. If we consider the Koszul resolution determined by the map $\phi'$ over $\cS$ then the sequence above is the degree $k$ strand of this Koszul resolution. 

For further details, see appendix A2 of \cite{Eisen} or for an approach via Schur complexes see Ch. 2 of \cite{Wey}. 

There exists a similar resolution of $k$-th exterior power of $\cF_2$ (by interchanging symmetric product and wedge product) which we will call \textit{$Sym-\wedge$ sequence of index $k$} associated to the given sequence. 

$$0 \rar Sym^k(\cF_0) \rar \cdots \rar Sym^{k-i}(\cF_0) \otimes \wedge^i \cF_1 \rar \cdots \rar \wedge^k \cF_1 \rar \wedge^k \cF_2 \rar 0$$

We will now prove a result from which Theorem \ref{result_main} will follow.

\begin{theorem} \label{thm_main} Let $E$ be any rank $k$ bundle on a smooth hypersurface $X \subset \p^{n+1}$ with $n \geq 2k + 1$. Assume further that $E$ satisfies the following two conditions,
\begin{enumerate}
\item $H^i_*(X, E) = 0, \,\, i \in \{1, 2, \ldots ,n-1\}$
\item $H^i_*(X, \wedge^m E) = 0, \,\, i = 2m-1,2m,\ldots,k+m \,\, \text{for each} \,\,m \in \{2,\ldots,k-1\}  $
\end{enumerate}

Then $E$ splits.
\end{theorem}

Despite the odd assumptions, the proof is very simple and we just use hypothesis of the theorem in $\wedge-Sym$ sequence for various indices to prove certain cohomological vanishings \eqref{S^m_G}, which is then used in a $Sym-\wedge$ sequence to prove the theorem.

\begin{proof}[Proof of theorem \ref{thm_main}] We write $\wedge-Sym$ sequence of some index $l \in \{2,\ldots,k\}$ for the short exact sequence \eqref{E_X_2}, 

$
\hspace{25mm} 0 \rar \wedge^l E(-d) \rar \wedge^{l-1} E(-d) \otimes F_1 \rar \cdots$

$\hspace{35mm} \cdots \rar E(-d) \otimes Sym^{l-1} F_1 \rar Sym^l F_1 \rar Sym^l G \rar 0$

This breaks up into short exact sequences,
$$
0 \rar \cG_{j-1,l} \rar \wedge^{l-j} E(-d) \otimes Sym^j F_1 \rar \cG_{j,l} \rar 0
$$
where $\cG_{0,l} = \wedge^l E(-d)$, $\cG_{j,l}$ is defined inductively for $j = 1, \ldots l-1$ and $\cG_{l,l} = Sym^l \,G$. For all $j \in \{0 ,\ldots l\}$, we claim $H^i_*(X,\cG_{j,l}) = 0$ for $i = 2l - j -1, 2l - j,\ldots k+l-j$. The case $j = 0$ is true by assumption in the theorem (putting $m = l$). For $j = t$, $$H^i(\wedge^{l-t} E(-d) \otimes Sym^t F_1) \rar H^i(\cG{t,l}) \rar  H^{i+1}(\cG_{t-1, l})$$

By induction, $H^{i+1}(\cG_{t-1,l}) = 0$ for $i+1 = \{2l-t, 2l-t+1, \ldots k+l-t+1\}$. This along with the assumption in the theorem and the fact that $F_1$ is split, proves the claim. Thus, 
\begin{eqnarray} \label{S^m_G}
H^i_*(X, Sym^l G) = 0 \,\,\text{for} \,\,i = l-1,l, \ldots,k
\end{eqnarray}

Now we look at $Sym-\wedge$ sequence of the index $k(=\text{rank}\,E)$ for the sequence \eqref{E_X_2}, 
$$
0 \rar Sym^k G \rar Sym^{k-1} G \otimes F_0 \rar \cdots G \otimes \wedge^{k-1} F_0 \rar \wedge^k F_0 \rar \wedge^k E \rar 0
$$
This breaks up into short exact sequences,
\begin{eqnarray} \label{M_k_defn}
0 \rar M_{j-1} \rar Sym^{k-j} G \otimes \wedge^j F_0 \rar M_j \rar 0
\end{eqnarray}
where $M_0 = Sym^k G$ and $M_j$ is defined inductively for $j = 1, \ldots k$ as 
$$M_j = \text{coker}(M_{j-1} \rar Sym^{k-j} G \otimes \wedge^j F_0)$$ 

\noindent Note that $M_k = \wedge^k E = \cO_X(e)$ for some $e \in \Z$. Using the vanishing given by \eqref{S^m_G} in sequence \eqref{M_k_defn} (and the fact that $F_0$ are split bundles), 
$$H^i_*(X, M_j) = 0 \,\,\text{for} \,\,i = k-j-1, k-j$$

\noindent Therefore the short exact sequence $0 \rar M_{k-1} \rar \wedge^k F_0 \rar \wedge^k E \rar 0$ splits. In particular, $M_{k-1}$ splits. This implies that the following sequence splits,
$$
0 \rar M_{k-2} \rar G \otimes \wedge^{k-1} F_0 \rar M_{k-1} \rar 0
$$
In particular, $G$ has a line bundle as a direct summand. Thus by lemma \ref{G_direct_summand}, $E$ splits.

\end{proof}

\begin{corollary} \label{cor_1} Let $E$ be a rank 3 ACM bundle on a smooth hypersurface $X$ with $\text{dim}(X) \geq 7$, then $E$ splits.
\end{corollary}
\begin{proof} We note that $\wedge^i E$ is ACM when $i = 1,2,3$. In particular, both the assumptions of theorem \ref{thm_main} are satisfied. Thus $E$ splits.
\end{proof}

\begin{corollary} \label{cor_2} Let $E$ be a rank 4 ACM bundle on a smooth hypersurface $X$ with $\text{dim}(X) \geq 9$ and $E$ (or any of its twists) admits a section with zero locus a complete intersection on $X$ of codimension 4, then $E$ splits.
\end{corollary}
\begin{proof} By theorem \ref{thm_main}, $E$ splits, if we can show that $H^i_*(X, \wedge^2 E) = 0$ for $i = 3,4,5,6$. Since $E$ splits $\Leftrightarrow E^{\vee}(m)$ splits for some $m \in \Z$, so we can assume that $E^{\vee}$ is globally generated and replace $E$ by $E^{\vee}$ (which is again rank 4 ACM). 

Suppose we are given any section $s \in H^0(X, E)$ such that the zero locus $Z(s)$ is a complete intersection of codimension 4 on $X$. This implies that there exists a resolution of $\cO_Z$ (see \cite{Eisen}, pp. 448),
$$
0 \rar \wedge^4 E \rar \wedge^3 E \rar \wedge^2 E \rar E \rar \cO_X \rar \cO_Z \rar 0
$$
We note that $Z$ is a complete intersection in $X$ (and hence in $\p^n$) of dimension $5$. In particular, $H^i_*(X, \cO_Z) = 0$ when $i = 1,2,3,4$. Using this along with the fact that $\wedge^i E$ is arithmetically Cohen-Macaulay for $i = 1,3,4$ we get that $H^i_*(X, \wedge^2 E) = 0$ for $i = 3,4,5,6$.
\end{proof}

\textit{Remark:} It is easy to verify the hypothesis of theorem \ref{thm_main} for any rank 2 ACM bundle when $n \geq 5$ which provides another proof for this well known splitting result.

%%%%%%%%%%%%%%%%%%%%%%%%%%%%%%%%%%%%%%%%%%%%%%%%%%%%%%%%%%%%
\section{Acknowledgement}
%%%%%%%%%%%%%%%%%%%%%%%%%%%%%%%%%%%%%%%%%%%%%%%%%%%%%%%%%%%%

We thank Suresh Nayak for outlining a proof for
the $\wedge-Sym$ and $Sym-\wedge$ sequences. We
thank Jishnu Biswas and G.V Ravindra for providing support and motivation. We thank
the referee for various feedbacks in improving the exposition.

%%%%%%%%%%%%%%%%%%%%%%%%%%%%%%%%%%%%%%%%%%%%%%%%%%%%%%%%%%%%

\end{document}